\documentclass[12pt]{amsart}
\usepackage{amssymb,latexsym}

\newcommand{\Q}{\mathbb{Q}}
\newcommand{\Z}{\mathbb{Z}}

\def\FAA{\Q \kern .05em \langle y_1,y_2,\dots \rangle}
\def\NC{\Z \kern .05em \langle y_1,y_2,\dots \rangle}

\theoremstyle{plain}
\newtheorem{theorem}{Theorem}[section]

\newtheorem{lemma}[theorem]{Lemma}
\newtheorem{proposition}[theorem]{Proposition}

\theoremstyle{remark}
\newtheorem{remark}{Remark}[section]
\newtheorem{example}[remark]{Example}

\newcommand{\vanish}[1]{}

\begin{document}
\title[Multiplicity Free Functions]
{Multiplicity free expansions of Schur $P$-functions}
\author{Kristin M. Shaw}
\email{krishaw@math.ubc.ca}
\author{Stephanie van Willigenburg}
\email{steph@math.ubc.ca}
\address{Department of Mathematics, Univ.
of British Columbia, Vancouver, BC  V6T 1Z2}
\thanks{Both authors were supported in part by the
National Sciences and Engineering Research Council of Canada.}


\begin{abstract}
After deriving inequalities on coefficients arising in the expansion of a
Schur $P$-function in terms of Schur functions we give criteria for when
such expansions are multiplicity free. From here we study the multiplicity
of an irreducible spin character of the twisted symmetric group in the product
of a basic spin character with an irreducible character of the symmetric
group, and determine when it is multiplicity free.
\end{abstract}
\maketitle
\section{Introduction \label{intro}}
In \cite{StembridgeII}, Stembridge determined when the product of two Schur functions is multiplicity free, which yielded when the outer product of characters of the symmetric groups did not have multiplicities. Meanwhile, in \cite{Bessenrodt}, Bessenrodt determined when the product of two Schur $P$-functions  is multiplicity free. This led to an analogous classification with respect to projective outer products of spin characters of double covers of the symmetric groups. In this article we interpolate between these two results to determine when a Schur $P$-function expanded in terms of Schur functions is multiplicity free. As an application we give criteria for when the multiplicity of an irreducible spin character of the twisted symmetric groups in the product of a basic spin character with an irreducible character of the symmetric groups is $0$ or $1$.

The remainder of this paper is structured as follows. We review the necessary definitions in the rest of this section. Then in Section 2 we derive some equalities and inequalities concerning certain coefficients. In Section 3 we give criteria for a Schur $P$-function to have a multiplicity free Schur function expansion before applying this to character theory in Section 4.

\subsection{Partitions}
A \emph{partition} $\lambda = \lambda _1\lambda _2\ldots \lambda _k$ of $n$ is a list of integers $\lambda _1\geq \lambda _2 \geq \ldots \geq \lambda _k >0$ whose sum is $n$, denoted $\lambda \vdash n$. We say $k$ is the \emph{length} of $\lambda$ denoted by $l(\lambda)$, $n$ is the \emph{size} of $\lambda$ and call the $\lambda _i$ \emph{parts}. Also denote the set of all partitions of $n$ by $P(n)$. Contained in $P(n)$ is the subset of partitions $D(n)$ consisting of all the partitions whose parts are distinct i.e. 
$\lambda _1> \lambda _2 > \ldots > \lambda _k >0$. We call such partitions \emph{strict}. A strict partition that will be of particular interest to us will be the \emph{staircase} (of length $k$):
$k(k-1)(k-2)\ldots 321$. Two other partitions that will be of interest to us are $\lambda +1^r$ and $\lambda \cup r$ for any given partition $\lambda$, and positive integer $r$. The partition 
$\lambda +1^r$ is formed by adding $1$ to the parts $\lambda _1, \lambda _2, \ldots , \lambda _r$ and $\lambda \cup r$ is formed by sorting the multiset of the union of parts of $\lambda $ and $r$. With these concepts in mind we are able to define a final partition that will be of interest to us, known as a near staircase. A partition  is a \emph{near staircase} if it is of the form $\lambda +1^r$, $1\leq r\leq k$ or $\lambda \cup r$, $r \geq k+1$ and $\lambda$ is the staircase of length $k$.

\begin{example}
$6321$ and $5421$ are both near staircases of $12$.
\end{example}

\subsection{Diagrams and tableaux}
For any partition $\lambda \vdash n$ the associated \emph{(Ferrers) diagram}, also denoted by $\lambda$, is an array of left justified boxes with $\lambda _i$ boxes in the $i$-th row, for $1\leq i \leq l(\lambda) $. Observe that in terms of diagrams a near staircase is more easily visualised as a diagram of a strict partition such that the deletion of exactly one row or column yields the diagram of a staircase. 

\begin{example}
The near staircases $6321$ and $5421$.
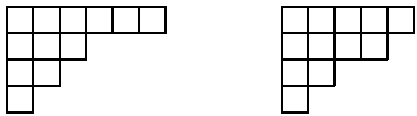
\begin{center}\begin{picture}(100,50)(-30,-5)
\put(0,0){\line(1,0){10}} 
\put(0,10){\line(1,0){20}}
\put(0,20){\line(1,0){30}} 
\put(0,30){\line(1,0){60}}
\put(0,40){\line(1,0){60}}

\put(0,40){\line(0,-1){40}} 
\put(10,40){\line(0,-1){40}}
\put(20,40){\line(0,-1){30}} 
\put(30,40){\line(0,-1){20}}
\put(40,40){\line(0,-1){10}}
\put(50,40){\line(0,-1){10}}
\put(60,40){\line(0,-1){10}}
\end{picture}
\begin{picture}(100,50)(-30,-5)
\put(0,0){\line(1,0){10}} 
\put(0,10){\line(1,0){20}}
\put(0,20){\line(1,0){40}} 
\put(0,30){\line(1,0){50}}
\put(0,40){\line(1,0){50}}

\put(0,40){\line(0,-1){40}} 
\put(10,40){\line(0,-1){40}}
\put(20,40){\line(0,-1){30}} 
\put(30,40){\line(0,-1){20}}
\put(40,40){\line(0,-1){20}}
\put(50,40){\line(0,-1){10}}
\end{picture}\end{center}
\end{example}  

Given a diagram $\lambda$ then the \emph{conjugate} diagram of $\lambda$, $\lambda '$, is formed by transposing the rows and columns of $\lambda$. The resulting partition $\lambda '$ is also known as the conjugate of $\lambda$. The \emph{shifted} diagram of $\lambda$, $S(\lambda)$ is formed by shifting the $i$-th row $(i-1)$ boxes to the right. If we are given two diagrams $\lambda$ and $\mu$ such that if $\mu$ has a box in the $(i,j)$-th position then $\lambda$ has a box in the $(i,j)$-th position then the \emph{skew} diagram $\lambda /\mu$ is formed by the array of boxes
$$\{ c| c\in \lambda , c\not\in \mu \}. $$ Now that we have introduced the necessary diagrams we are now in a position to fill the boxes and form tableaux.

Consider the alphabet
$$1'<1<2'<2<3'<3< \ldots$$
For convenience we call the integers $\{ 1,2,3,\ldots \}$ \emph{unmarked} and the integers $\{ 1',2',3',\ldots \}$ \emph{marked}. Any filling of the boxes of a diagram $
\lambda$ with letters from the above alphabet is called a \emph{tableau} of shape $\lambda$. If we fill the boxes of a skew or shifted diagram we similarly obtain a skew or shifted tableau. Given any type of tableau, $T$, we define the \emph{reading word} $w(T)$ to be the entries of $T$ read from right to left and top to bottom, and define the \emph{augmented reverse reading word} $\hat{w}(T)$ to be $w(T)$ read backwards with each entry increased by one according to the total order on our alphabet e.g. if $T= \begin{array}{ll}
1'&1\\
1&2'\\
2&\\
\end{array}$ then $w(T)=11'2'12$ and $\hat{w}(T)=3'2'212'$. When there is no ambiguity concerning the tableau under discussion we refer to the reading word and augmented reverse reading word as $w$ and $\hat{w}$ respectively. We also define the content of $T$, $c(T)$, to be the sequence of integers $c_1 c_2 \ldots $ where 
$$c_i = |i|+|i'|$$
and $|i|$ is the number of $i$s in $w(T)$ and $|i'|$ is the number of $i'$s in $w(T)$. For our previous example $c(T)=32$.
Given a word we say it is \emph{lattice} if as we read it if the number of $i$s we have read is equal to the number of $(i+1)$s we have read then the next symbol we read is neither an $(i+1)$ nor an $(i+1)'$ e.g. $11'2'23$ is lattice, however, $11'22'3$ is not.

Let $T$ be a (skew or shifted) tableau, then we say $T$ is \emph{amenable} if it satisfies the following \cite[p259]{MacD}:
\begin{enumerate}
\item The entries in each row of $T$ weakly increase.
\item The entries in each column of $T$ weakly increase.
\item Each row contains at most one $i'$ for $i\geq 1$.
\item Each column contains at most one $i$ for each $i\geq 1$.
\item The word $w\hat{w}$ is lattice.
\item In $w$ the rightmost occurrence of $i$ is to the right of the rightmost occurrence of $i'$ for all $i$.
\end{enumerate}

\begin{example}
The first tableau is amenable whilst the second is not as it violates the lattice condition.
$$\begin{array}{ll}
1'&1\\
1&2'\\
2&
\end{array}
\qquad
\begin{array}{ll}
1'&1\\
1&2\\
2&
\end{array}
$$
\end{example}

Before we define Schur $P$-functions we make two observations about amenable tableaux.

\begin{lemma}\label{no_high_row}
Let $T$ be an amenable tableau then if $i$ or $i'$ appear in row $j$ then $j\geq i$.
\end{lemma}

\begin{proof}
We proceed by induction on the number of rows of $T$. If $T$ has one row then the result is clear. Assume the result holds up to row $(k-1)$. Consider row $k$. If it has an entry in it greater than $k$, $l$ or $l'$, then it must lie in the rightmost box since the rows of $T$ weakly increase. However, this ensures that $w\hat{w}$ is not lattice as when we first read $l$ or $l'$ in $w$ we will have read no $(l-1)$  or $l$.
\end{proof}

\begin{lemma}\label{increasing_primes}
Let $T$ be an amenable tableau with $c(T)=k(k-1)(k-2)\ldots (k-j)$ then in $w(T)$
$$|i'|\leq |(i+1)'| \qquad 1\leq i\leq j.$$
\end{lemma} 

\begin{proof}
To prove this we consider the lattice condition on $w\hat{w}$. Assume $|i'| > |(i+1)'| $ then as we read $w$ there will be a rightmost occurrence  when $|i|=|(i+1)|$. However, because of $c(T)$ before we read another $i$ in $w\hat{w}$ we must read $(i+1)$ or $(i+1)'$.
\end{proof}

\subsection{Schur $P$-functions}
Given commuting variables $x_1, x_2, x_3, \ldots $ let the $r$-th elementary symmetric function, $e_r$, be defined by
$$e_r=\sum _{i_1<i_2<\ldots <i_r}x_{i_1}x_{i_2}\ldots x_{i_r}.$$
Moreover, for any partition $\lambda = \lambda _1\lambda _2\ldots \lambda _k$ let
$$e_\lambda = e _{\lambda _1}e_{\lambda _2}\ldots e_{\lambda _k}$$
then the algebra of symmetric functions, $\Lambda$, is the algebra over $\mathbb{C}$ spanned by all $e_\lambda$ where $\lambda \vdash n, n>0$ and $e_0=1$. Another well-known basis for $\Lambda$ is the basis of \emph{Schur functions}, $s _\lambda$, defined by
$$s _\lambda = \det (e_{\lambda _i'-i+j})_{1\leq i,j\leq k}$$
where $e_r =0$ if $r<0$. It is these that can be used to define the subalgebra of Schur $P$-functions, $\Gamma$. More precisely, let $\lambda$ be a strict partition of $n>0$, then $\Gamma$ is spanned by $P_0 = 1$ and all
$$P_\lambda = \sum _{\mu\vdash n} g _{\lambda\mu} s_\mu$$
where $g_{\lambda \mu }$ is the number of amenable tableaux $T$ of shape $\mu$ and content $\lambda$.

Amenable tableaux can also be used to describe the multiplication rule for the $P_\lambda$ as follows. Let $\lambda, \mu ,\nu$ be strict partitions then
$$P_\mu P_\nu = \sum f^\lambda _{\mu\nu} P_\lambda $$ 
where $f^\lambda _{\mu\nu}$ is the number of amenable shifted skew tableaux of shape $\lambda /\mu$ and content $\nu$. Further details on Schur and Schur $P$-functions can be found in \cite{MacD}.

\section{Relations on Stembridge coefficients}
The combinatorial descriptions of the $g _{\lambda\mu}$ and $f ^\lambda _{\mu\nu}$ from the previous section were discussed by Stembridge \cite{Stembridge} who also implicitly observed the following useful relationship between them.

\begin{lemma}\label{f_and_g}
If $\lambda \in D(n)$, $\mu \in P(n)$, and $\nu$ is a staircase of length $l(\mu)$ then $$f^{\mu + \nu}_{\lambda \nu} = g_{\lambda \mu}.$$
\end{lemma}

\begin{proof}
By definition $f^{\mu + \nu}_{\lambda \nu}$ is the number of amenable shifted skew tableaux of shape
$\mu + \nu / \lambda$ and content $\nu$. However, since $P_\lambda P_\nu = P_\nu P_\lambda$ we know $f^{\mu + \nu}_{\lambda \nu}$ is also the number of amenable shifted skew tableaux of shape
$\mu + \nu / \nu$ and content $\lambda$. In addition, the shifted skew diagram $\mu + \nu / \nu$ is simply the diagram $\mu$. Thus $f^{\mu + \nu}_{\lambda \nu}$ is the number of amenable tableaux of shape $\mu$ and content $\lambda$, and this is precisely $ g_{\lambda \mu}$. \end{proof}

There are also equalities between the $g _{\lambda\mu}$.

\begin{lemma}\label{g_and_conjugate}
If $\lambda \in D(n)$ and $\mu \in P(n)$ then
$$g_{\lambda \mu}=g_{\lambda \mu '}.$$\end{lemma}

\begin{proof}
Let $\omega : \Lambda \rightarrow \Lambda$ be the  involution on symmetric functions such that $\omega(s_\lambda) = s_{\lambda '}$.  In Exercise 3 \cite[p259]{MacD} it was proved  that $\omega (P_{\lambda}) = P_{\lambda}$.  Consequently, 
$$\omega (P_{\lambda}) = \sum g_{\lambda \mu} \omega (s_{\mu}) = \sum g_{\lambda \mu} s_{\mu '}$$ 
and the result follows. \end{proof}
 
Finally, we present two inequalities that will be useful in the following sections and relate amenable tableaux of different shape and content.

\begin{lemma}\label{ineq}

Given $\lambda \in D(n)$ and $\mu \in P(n)$, if $r \leq l(\lambda)+1$ and $s \geq \lambda_1+1$ then
$$g_{\lambda \mu} \leq g_{(\lambda + 1^r )(\mu + 1^r)}$$
and 
$$
g_{\lambda \mu} \leq g_{(\lambda \cup s) (\mu \cup s)}. 
$$
\end{lemma}

\begin{proof}
Consider an amenable tableau $T$ of shape $\mu$ and content $\lambda$.
To prove the first inequality, for $1 \leq i \leq r$ append a box containing $i$ to row $i$ on the right side of $T$. By Lemma \ref{no_high_row} it is straightforward to verify that this is an amenable tableau of shape $\mu + 1^r$ and content $\lambda +1^r$. For the second inequality, replace each entry $i$ with $(i+1)$ and each entry $i '$  with $(i+1)'$ for $1 \leq i \leq l(\lambda)$ to form $T'$. Then append a row of $s$ boxes each containing $1$ to the top of $T'$. Again, it  is straightforward to check this is an amenable tableau of shape $ \mu \cup s$ and content $\lambda \cup s$.\end{proof}

\begin{example}
Consider the amenable tableau $\begin{array}{lcc}
1'&1&1\\
1'&2&2\\
1&&\end{array}$. Then the following two tableaux illustrate the operations utilised in proving the first and second inequalities, respectively.
$$
\begin{array}{llll}
1'&1&1&\mathbf{1}\\
1'&2&2&\mathbf{2}\\
1&\mathbf{3}&&
\end{array}
\qquad
\begin{array}{llllll}
\mathbf{1}&\mathbf{1}&\mathbf{1}&\mathbf{1}&\mathbf{1}&\mathbf{1}\\
{2'}&2&2&&&\\
{2'}&3&3&&&\\
{2}&&&&&
\end{array}$$
\end{example}

\section{Multiplicity free Schur expansions}
Despite the number of conditions amenable tableaux must satisfy, it transpires that most Schur $P$-functions do not have multiplicity free expansions in terms of Schur functions.

\begin{example}
Neither $P_{541}$ nor $P_{654}$ is multiplicity free. We see this in the first case by observing there are at least two amenable tableaux of shape $4321$ and content $541$:
$$\begin{array}{llll}
1'&1&1&1\\
1&2'&2&\\
2&2&&\\
3&&&
\end{array}
\qquad
\begin{array}{llll}
1'&1&1&1\\
1&2'&2&\\
2'&3&&\\
2&&&
\end{array}.$$
In the second case we note that there are at least two amenable tableaux of shape $54321$ and content $654$:
$$\begin{array}{lllll}
1'&1&1&1&1\\
1&2'&2&2&\\
2'&3'&3&&\\
2&3'&&&\\
3&&&&
\end{array}
\qquad
\begin{array}{lllll}
1'&1&1&1&1\\
1&2'&2&2&\\
2&2&3'&&\\
3'&3&&&\\
3&&&&
\end{array}.$$
\end{example}

However, in some cases it is easy to deduce a certain Schur $P$-function expansion is multiplicity free as we can give a precise description of it in terms of Schur functions.

\begin{proposition}\label{n_and_hook}
$$P_n = \sum _{1\leq k \leq n} s_{k1^{n-k}}$$
and
$$
P_{(n-1)1} = \sum _{1\leq k \leq n-1} s_{k1^{n-k}} + \sum _{2\leq k \leq n-2} s_{k21^{n-k-2}}.$$
\end{proposition}

\begin{proof}
For the first result observe that $g_{n(k1^{n-k})}=1$ since the tableau filled with one $1'$ and $(k-1)$ $1$s in the first row and one $1$ and $(n-k)$ $1'$s in the first column is the only amenable tableau of shape $k1^{n-k}$ and content $n$. Then observe that $g_{n\mu}=0$ for any other $\mu$ since we have no way to fill a $2\times 2$ rectangle with only $1$ or $1'$ and create an amenable tableau.

For the second result note that since the content of any tableau we create is $(n-1)1$ we will be filling our diagram $(n-1)$ $1$ or $1'$s and one $2$. As in the previous case if our resulting tableau is to be amenable the $1$s must appear in the first row and column marked or unmarked as necessary. The unmarked $2$ can now only appear in one of two places, either at the end of the second row, or the end of the first column.
\end{proof}

A third multiplicity free expansion is obtained from the determination of when a Schur $P$-function is equal to a Schur function.

\begin{theorem}\label{P_equal_schur}
$$P_\lambda = s_\lambda \mbox{ if and only if }\lambda \mbox{ is a staircase.}$$\end{theorem}

\begin{proof}
The reverse implication is proved in Exercise 3(b)  \cite[p 259]{MacD}. For the forward implication assume that $\lambda = \lambda _1 \ldots \lambda _k \in D(n)$ is not a staircase. It follows there must exist at least one $1\leq i \leq k$ for which $\lambda _i \geq \lambda _{i+1}+2$. We are going to show that in this situation there are at least two amenable tableaux with content $\lambda$.

Consider the tableau, $T$, of shape and content $\lambda$ where the $j$-th row is filled with unmarked $j$s. Clearly $T$ is amenable. Now consider the first row $i$ for which $\lambda _i \geq \lambda _{i+1}+2$. Delete the rightmost box from this row and append it to the first column of $T$ to form a tableau $T'$ of shape $\lambda _1 \lambda _2 \ldots \lambda _i -1 \ldots \lambda _k1$ and content $\lambda$. Now alter the entries in the first column of $T'$ as follows. In row $i$ change $i$ to $i'$ and in rows $ i+1\leq j\leq k$ change $j$ to $j-1$. Finally in row $k+1$ change $i$ to $k$. Now $T'$ is an amenable tableau, and we are done.
\end{proof}

As we will see, staircases play an important role in the determination of multiplicity free expansions of Schur $P$-fuctions. 

\begin{theorem}\label{multfree}
For $\lambda \in D(n)$ the Schur function expansion of $P_\lambda$ is multiplicity free if and only if $\lambda$ is one of the following
\begin{enumerate}
\item staircase
\item near staircase
\item $k(k-1)(k-2) \dots 4 3$
\item $k(k-1).$
\end{enumerate}
\end{theorem}

\begin{proof}
If $\lambda \in D(n)$ is not one of the partitions listed in Theorem \ref{multfree} then it must satisfy one of the following:

\begin{enumerate}
\item For all $ 1 \leq i < l(\lambda)$, $\lambda_i = \lambda_{i+1} +1$ and $\lambda_{l(\lambda)} \geq 4$ and $l(\lambda)\geq 3$.
\item There exists exactly one $1< i < l(\lambda)$ such that $\lambda_ i \geq \lambda_{i+1} + 3$ and $\lambda_j = \lambda_{j+1} +1$ for all $ 1 \leq j < l(\lambda), j\neq i$ and $\lambda _{l(\lambda)}=1$.
\item There exists $i < j$ such that $ \lambda_i \geq \lambda_{i+1} +2$  and $\lambda_j \geq \lambda_{j+1} +2$ and $\lambda _{l(\lambda)}=1$.
\item There exists $1\leq i < l(\lambda)$ such that $\lambda _i\geq \lambda _{i+1}+2$ and $\lambda _{l(\lambda)}\geq 2$.
\end{enumerate}

If $\lambda$ satisfies the first criterion then by observing $P_{654}$  is not multiplicity free and Lemma \ref{ineq}, it follows that $P_\lambda$ is not multiplicity free. Similarly, if $\lambda$ satisfies the second criterion then by observing that $P_{541}$ has multiplicity  and Lemma \ref{ineq}, again $P_\lambda$ has  multiplicity.  If $\lambda$ satisfies the third criterion then consider the partition $\nu = k(k-2)(k-3)\ldots 431 \in D(n)$ and the partition $\mu= (k-1)(k-2)(k-3)\ldots 432 \in D(n)$  for $k\geq 5$. We now show that $g_{\nu\mu} >1$. Take a diagram $\mu$ and fill the first row with one $1'$ and $(k-2)$ $1$s. For $i>1$ fill the $i$-th row with one $(i-1)$ and the rest $i$s. This is clearly an amenable tableau $T$. If we now change the $(k-3)$ to a $(k-3)'$ in the second column of the penultimate row of $T$ we obtain another amenable tableau of shape $\mu$ and content $\nu$. Thus $g_{\nu\mu}>1$ and $P_\nu$ is not multiplicity free. This combined with Lemma \ref{ineq} yields that $P_\lambda$ is not multiplicity free if $\lambda$ satisfies the third criterion. Lastly, if $\lambda$ satisfies the fourth criterion then consider the partition $\nu = k(k-2)(k-3)\ldots 432 \in D(n)$ and the staircase $\mu= (k-1)(k-2)(k-3)\ldots 4321 \in D(n)$  for $k\geq 4$. We now prove that $g_{\nu\mu} >1$. Take the staircase $\mu$ and fill the first row with one $1'$ and $(k-2)$ $1$s. For $1<i<k-1$ fill the $i$-th row with one $(i-1)$ and the rest $i$s. Fill the last row with one $(k-2)$. It is straightforward to see that this is  an amenable tableau $T$. If we now change the $(k-2)$ to a $(k-2)'$ in the second column of the penultimate row of $T$ we obtain another amenable tableau of shape $\mu$ and content $\nu$ and $g_{\nu\mu}>1$ as desired. Consequently, $P_\nu$ is not multiplicity free, which combined with Lemma \ref{ineq} shows that $P_\lambda$ is not multiplicity free if $\lambda$ satisfies the fourth criterion.

Finally it remains to show that if $\lambda\in D(n)$ is one of the partitions listed in Theorem \ref{multfree} then $g_{\lambda\mu}\leq 1$ for all $\mu$. If $\lambda$ is a staircase the result follows from Theorem \ref{P_equal_schur}. If $\lambda$ is a near staircase of the form $mk(k-1)\ldots 321$ then by Lemma \ref{increasing_primes} it follows that any amenable tableau of content $\lambda$ must contain no $i'$ for $i>1$ and thus there can exist at most one amenable tableau of shape $\mu$ and content $\lambda$ for any given $\mu$. Consequently $g_{\lambda\mu}\leq 1$ for all $\mu$. Similarly if $\lambda$ is the other type of near staircase or of the form $k(k-1)\ldots 43$ then by Lemma \ref{f_and_g} we can calculate $g_{\lambda\mu}$ by enumerating all amenable shifted skew tableaux, $T$, of shape $\mu + \nu /\lambda$ and \emph{staircase} content. By Lemma \ref{increasing_primes} it follows that no entry in $T$ can be marked. {}From this we can deduce  that  there can exist at most one amenable shifted skew tableau and so $g_{\lambda\mu}\leq 1$. Finally a proof similar to that of Proposition \ref{n_and_hook} yields that $P_{k(k-1)}$ is multiplicity free.\end{proof}

\begin{remark} The reverse direction of the above theorem can also be proved via Lemma \ref{f_and_g} and \cite[Theorem 2.2 ]{Bessenrodt}.\end{remark}

\section{Multiplicity free spin character expansions} 
The twisted symmetric group $\tilde{S} _n$ is presented by
$$\langle z, t_1, t_2, \ldots , t_{n-1}|\ z^2=1, t_i^2=(t_it_{i+1})^3=(t_it_j)^2=z \ |i-j|\geq 2\rangle .$$
Moreover, the ordinary representations of $\tilde{S} _n$ are equivalent to the projective representations of the symmetric group $S_n$ and in \cite{Stembridge} Stembridge determined the product of a basic spin character of $\tilde{S} _n$ with an irreducible character of $S_n$, whose description we include here for completeness. If $\lambda =\lambda _1\ldots \lambda _k \in D(n)$ then define
$$\varepsilon _\lambda =\left\{ \begin{array}{ll}
1&\mbox{ if $n-k$ is even}\\
\sqrt{2}&\mbox{ if $n-k$ is odd} \end{array}\right. .$$
Let $\phi ^\lambda$ be an irreducible spin character of $\tilde{S} _n$, $\chi ^\mu$ for $\mu\in P(n)$ be an irreducible character of $S_n$ and $\langle\cdot , \cdot\rangle$ be defined on $\Lambda$ by $\langle s_\mu, s_\nu \rangle =\delta _{\mu\nu}$ then we have

\begin{theorem}\cite[Theorem 9.3]{Stembridge}
If $\lambda \in D(n)$, $\mu \in P(n)$ then
\begin{equation}\langle \phi ^n\chi ^\mu , \phi ^\lambda \rangle = \frac{1}{\varepsilon _\lambda\varepsilon _n} 2^{(l(\lambda)-1)/2}g_{\lambda\mu}\label{char}\end{equation}unless $\lambda = n$, $n$ even, and $\mu = k 1^{n-k}$ in which case the multiplicity is $0$ or $1$.
\end{theorem}

Using this formula we can deduce

\begin{theorem}
If $\lambda\in D(n), \mu\in P(n)$ then the coefficient of $\phi ^\lambda$ in $\phi ^n\chi ^\mu$ is multiplicity free for all $\mu$ if and only if $\lambda$ is one of the following
\begin{enumerate}
\item $n$
\item $(n-1)1$
\item $k(k-1)$
\item $(2k+1)21$
\item $543$
\item $431$.
\end{enumerate}
\end{theorem}

\begin{proof}
Considering Equation \ref{char} we first show no $\lambda$ exists such that $g _{\lambda \mu}\geq 2$ but $\langle \phi ^n\chi ^\mu , \phi ^\lambda \rangle$ is multiplicity free. If such a $\lambda$ did exist then
$$2^{(l(\lambda)-1)/2}<\varepsilon _\lambda \varepsilon _n$$
where $\varepsilon _\lambda \varepsilon _n = 1, \sqrt{2}, 2$ depending on $\lambda$ and its size. However, if $\varepsilon _\lambda \varepsilon _n = 1$ then $l(\lambda) =0$ and if $\varepsilon _\lambda \varepsilon _n = \sqrt{2}$ then $l(\lambda) = 1$ so $\lambda = n$ but then $\varepsilon _\lambda \varepsilon _n \neq \sqrt{2}$ so we must have that 
$$2^{(l(\lambda)-1)/2}<2$$
and hence $l(\lambda) < 3$. Since $\varepsilon _\lambda , \varepsilon _n = \sqrt{2}$ it follows $n$ is even and we must in fact have $\lambda = n$. By Proposition \ref{n_and_hook} we know in this case $g _{\lambda\mu}\leq 1$ for all $\mu$ and we have our desired contradiction. Consequently if $\langle \phi ^n\chi ^\mu , \phi ^\lambda \rangle$ is multiplicity free then 
$g _{\lambda\mu}\leq 1$. Additionally we must have 
$2^{(l(\lambda)-1)/2}=\varepsilon _\lambda \varepsilon _n$ and so it remains for us to check three cases.
\begin{enumerate}
\item $\varepsilon _\lambda \varepsilon _n =1$: We have $l(\lambda)=1$ and so by Theorem \ref{multfree} $\lambda = n$, $n$ odd.
\item $\varepsilon _\lambda \varepsilon _n =\sqrt{2}$: We have $l(\lambda)=2$ and so by Theorem \ref{multfree} $\lambda = (n-1)1$, or $\lambda = k(k-1)$.
\item $\varepsilon _\lambda \varepsilon _n ={2}$: We have $l(\lambda)=3$ and since $\varepsilon _n =\sqrt{2}$ it follows that $n$ is even. Hence by Theorem \ref{multfree} $\lambda = (2k+1)21$, $\lambda = 543$, or $\lambda = 431$.
\end{enumerate}
\end{proof}
\providecommand{\bysame}{\leavevmode\hbox to3em{\hrulefill}\thinspace}\

\section*{Acknowledgements} The authors would like to thank the referee for their valuable comments.

\end{document}